\documentclass[a4paper]{article}
\usepackage{amsthm,amsmath,amssymb,vmargin}
\usepackage{comment}
\usepackage{tikz}
\usetikzlibrary{snakes}
\tikzstyle{empty}=[circle,draw=black!80,thick]
\tikzstyle{nero}=[circle,draw=black!80,fill=black!80,thick]
\usepackage{caption}
\title{\textbf{A new lower bound on the vertices of Specht modules for symmetric groups}}                    
\author{Eugenio Giannelli}

\newtheorem{teo1.2}{Theorem 1.2}[section]
\newtheorem{teo1.1}{Theorem 1.1}[section]
\newtheorem{teo}{Theorem}[section]
\newtheorem{theorem}{Theorem}

\newtheorem{lemma}[teo]{Lemma}

\newtheorem{osserv}[teo]{Remark}

\newtheorem{esem}[teo]{Example}

\theoremstyle{definition}

\begin{document}

\begin{abstract}
In this paper we study the vertices of indecomposable Specht modules for symmetric groups. For any given indecomposable non-projective Specht module,
the main theorem of the article describes a family of $p$-subgroups contained in its vertex. 
The theorem generalizes and improves an earlier result due to Wildon in \textit{Vertices of {S}pecht modules and blocks of the symmetric group,}
J. Alg 323 (2010) 2243--2256.
\end{abstract}

\maketitle

\section{Introduction}
One of the mainstream themes in the representation theory of finite groups has been to determine global information about the $p$-modular representation theory of a group $G$ by studying its local structure, namely representations of its $p$-subgroups and their normalizers.
An interesting topic is the investigation of the vertices in the sense defined by Green in \cite{Green}, of specific indecomposable modules such as simple modules. 
In the case of the symmetric group the study of the modular structure of Specht modules and simple modules is one of the important open problems in the area. These two families of modules are closely connected and the comprehension of properties of one of those often gave new results about the structure of the other. 

The vertices of Specht modules were first considered by Murphy and Peel in \cite{MurPeel}; their work focused on hook Specht modules in the case $p=2$. In \cite{MW1}, Wildon made some progress on the topic by characterizing the vertices of hook Specht modules for $\mathbb{F}S_n$ when $\mathbb{F}$ is a field of prime characteristic $p$ and $n$ is not divisible by $p$. Muller and Zimmermann described vertices and sources of some hook Specht and simple modules in 
\cite{MULZIM}. In \cite{KayJin} Lim gave a necessary condition for a Specht module to have an abelian vertex and characterized the possible abelian vertices of Specht modules. 
The vertices of irreducible Specht modules are completely described by the work of Hemmer \cite{hemmer} and Donkin \cite{Donkin}. In particular, in \cite{hemmer} it is shown that every irreducible Specht module is a signed Young module and in \cite{Donkin} a complete characterization of the vertices of signed Young modules is given.
Danz, K\"{u}lshammer and Zimmermann in \cite{DKZ}, characterized the vertices of all the simple modules of degree at most $14$ and $15$ in characteristic $2$ and $3$ respectively.
Wildon gave a general structural description of the vertices of all Specht modules; in \cite{MW2} he showed that a vertex of a non-projective Specht module contains some large $p$-subgroup of $S_n$.
In this work we will generalize and improve the lower bound on the vertex given in \cite[Theorem 1.1]{MW2} for Specht modules $S^\lambda$ defined over any field of prime characteristic $p$. In particular we will prove the following theorem, involving the concept of splitting partition and the definition of the group $L_{\mu}(t)$; these are the key ideas of the article and will be given in full detail in Section 3. 

\begin{theorem}\label{T1}
Let $n$ be any natural number and let $\mathbb{F}$ be a field of prime characteristic $p$. Let $\lambda$ be a partition of $n$, $\mu$ a splitting partition of $\lambda$ and $t$ a $\lambda$-tableau. If the Specht module $S^\lambda$ defined over $\mathbb{F}$ is indecomposable, then its vertex contains a subgroup conjugated to a Sylow $p$-subgroup of $L_{\mu}(t)$. 
\end{theorem}

The paper is structured as follows. In Section 2 we firstly recall the basic notions from modular representation theory of finite groups and we describe the Brauer homomorphism, this is our main tool to calculate vertices. In the second part of Section 2 we give a brief summary of the main properties of Specht modules for the symmetric group. Section 3 is mainly devoted to the proof of Theorem \ref{T1}. 
We show that Theorem \ref{T1} improves the bound given by Wildon in \cite{MW2} and we conclude by remarking that when $\lambda$ is an hook partition of a natural number $n$ not divisible by $p$, then the lower bound obtained from Theorem \ref{T1} is attained.

\section{Preliminaries}

\subsection{The Brauer homomorphism}
Let $G$ be a finite group and $\mathbb{F}$ a field of prime characteristic $p$. Let $H$ be a subgroup of $G$, an $\mathbb{F}G$-module $M$ is called \textit{$H$-projective} if there exists an $\mathbb{F}H$-module $N$ such that $M$ is a summand of the induced module $N\uparrow_H^G$. A \textit{vertex} $Q$ of an indecomposable $\mathbb{F}G$-module $M$ is a subgroup of $G$ that is minimal with respect to the condition that $M$ is relatively $Q$-projective. By \cite[page 435]{Green}, we have that the vertices of $M$ are $p$-groups and they are all conjugate in $G$. Moreover an indecomposable $\mathbb{F}Q$-module $N$ such that $M$ is a summand of $N\uparrow_Q^G$ is called a \textit{source} of $M$. By \cite[Page 66]{Alperin}, we have that the source of $M$ is unique up to conjugation in $N_G(Q)$.

Let now $V$ be an $\mathbb{F}G$-module.
Given a $p$-subgroup $Q\leq G$ we denote by $V^Q$ the set
$$V^Q=\{v\in V : vQ=v\}$$ of $Q$-fixed elements. It is easy to see that $V^Q$ is an 
$\mathbb{F}N_G(Q)$-module on which $Q$ acts trivially. 
For $R$ a proper subgroup of $Q$, the \emph{relative trace} map 
$\mathrm{Tr}_R^Q:V^R\rightarrow V^Q$ is the linear map defined by 
$$\mathrm{Tr}_R^Q(v)=\sum_g vg,$$
where the sum is over a set of right coset 
representatives for $R$ in $Q$. 
We observe that 
$$ \mathrm{Tr}^Q(V):=\sum_{R<Q}\mathrm{Tr}_R^Q(V^R) $$
is an $\mathbb{F}N_G(Q)$-module contained in $V^Q$. Moreover, for all $R<P<Q$ we have that $$\mathrm{Tr}_R^Q(V^R)=\mathrm{Tr}_P^Q(\mathrm{Tr}_R^P(V^R)).$$ 
Therefore $$ \mathrm{Tr}^Q(V)=\sum_{R\in\Omega_Q}\mathrm{Tr}_R^Q(V^R),$$
where $\Omega_Q$ is the set containing all maximal subgroups of $Q$.

The \emph{Brauer correspondent} of $V$ with respect to $Q$
is the $FN_G(Q)$-module $V(Q)$ defined by
$$V(Q) = V^Q / \sum_{R<Q}\mathrm{Tr}_R^Q(V^R). $$
For the scope of this paper it is very important to remark that if $V$ is an indecomposable $\mathbb{F}G$-module and $Q\leq_p G$, then $V(Q)\neq 0$ implies that $Q$ is contained in a vertex of $V$ (see \cite[(1.3)]{Broue}).

\subsection{The representation theory of $S_n$}
Let $\lambda=(\lambda_1,\ldots,\lambda_k)$ be a partition of a natural number $n$. The Young diagram of $\lambda$ is an array of boxes, left aligned, having $\lambda_j$ boxes in the $j^{th}$-row for all $j\in\{1,\ldots,k\}$. A $\lambda$-tableau is an assignment of the numbers $\{1,2,\ldots,n\}$ to the boxes of the Young diagram of $\lambda$ such that no number appears twice. 
The symmetric group $S_n$ acts naturally on the set of $\lambda$-tableaux by permuting the entries within the boxes.
We call \textit{row-standard} any $\lambda$-tableau having the entries of each row ordered increasingly from left to right. Similarly a $\lambda$-tableau is called \textit{column-standard} if the entries of each column are increasingly ordered from top to bottom. When a $\lambda$-tableau is both row-standard and column-standard is called \textit{standard}. Given a $\lambda$-tableau $v$ we will denote by $\overline{v}$ the row-standard tableau obtained from $v$ by sorting its rows in increasing order. We will call $\overline{v}$ the \textit{row-straightening of} $v$. 
There is a natural ordering on the set of standard $\lambda$-tableaux known as the \textit{dominance order} (see \cite[Section 3.1]{MW2}). Given two standard $\lambda$-tableaux $t$ and $u$ we write $t\triangleright u$ to express that $t$ dominates $u$.

We turn now to a brief account of the theory of Specht modules for $S_n$, we refer the reader to \cite{James} for further details and examples. 
We say that two $\lambda$-tableaux $t$ and $u$ are row-equivalent if the entries in each row of $t$ are the same as the entries in the corresponding row of $u$. It is easy to see that this defines an equivalence relation on the set of $\lambda$-tableaux. We will denote by $\{t\}$ the row-equivalence class of $t$ and we will say that $\{t\}$ is a \textit{$\lambda$-tabloid}. The symmetric group $S_n$ acts naturally on the set of $\lambda$-tabloids, 
therefore we can define $M^\lambda$ to be the $S_n$-permutation representation generated as a vector space by the set of all $\lambda$-tabloids.  
Given any $\lambda$-tableau $t$ we denote by $C(t)$ the column stabilizer of $t$, namely the subgroup of $S_n$ that fixes the columns of $t$ setwise. 
The $\lambda$-\textit{polytabloid} corresponding to the $\lambda$-tableau $t$ is the following element of $M^\lambda$: 
$$e_t=\sum_{g\in C(t)}\mathrm{sgn}(g)\{t\}g.$$
The Specht module $S^\lambda$ is the submodule of $M^\lambda$ linearly generated by the polytabloids. 
Notice that for all $h\in S_n$ we have that $e_th=e_{th}$ for any given $\lambda$-tableau $t$. Moreover if $g\in C(t)$ then it is easy to observe that 
$e_tg=\mathrm{sgn}(g)e_t$. Finally we will say that $e_t$ is a \textit{standard polytabloid} if $t$ is a standard $\lambda$-tableau. 
The main theorem about the structure of Specht modules is the following Standard Basis Theorem proved by James in \cite{James}. Here we will present it in its stronger version with the contribution made by Wildon in \cite[Proposition 4.1]{MW2}.
\begin{teo}\label{Standard}
The set of standard $\lambda$-polytabloids is a $\mathbb{Z}$-basis for the Specht module $S^\lambda$ defined over the ring of integer numbers. Moreover if $v$ is a column-standard $\lambda$-tableau then its row-straightening $\overline{v}$ is a standard $\lambda$-tableau and $$e_v=e_{\overline{v}}+x,$$
where $x$ is a $\mathbb{Z}$-linear combination of standard $\lambda$-polytabloids $e_t$ such that $\overline{v}\triangleright t$. 
\end{teo}

\subsection{Wildon's result}
As explained in the introduction, this paper is devoted to the generalization of Theorem 1.1 of \cite{MW2}. 

For any partition $\lambda$ of $n$ and any $\lambda$-tableau $t$  we will denote by $R(t)$ the subgroup of $S_n$ that stabilizes the rows of $t$ setwise.
We define $H(t)$ to be the subgroup of $R(t)$ which permutes, as blocks for its action, the columns of equal length in $t$.

For the reader convenience we restate below Theorem 1.1 of \cite{MW2}.

\begin{teo}\label{mw}
Let $\lambda$ be a partition of $n$ and let $t$ be a $\lambda$-tableau. If the Specht module $S^\lambda$, defined over a field of characteristic $p$, is indecomposable, then it has a vertex containing a subgroup isomorphic to a Sylow $p$-subgroup of $H(t)$.
\end{teo}

\vspace{.2cm}

For example, if $\lambda=(5,5,2,2,2,2)$ and $t$ is the $\lambda$-tableau shown in Figure \ref{fig:a} below,

\begin{figure}[!htp]
\begin{center}

\begin{tikzpicture}[scale=0.3, every node/.style={transform shape}]
\tikzstyle{every node}=[font=\large]

\draw (0,12)--(10,12);
\draw (0,10)--(10,10);
\draw (0,8)--(10,8);
\draw (0,6)--(4,6);
\draw (0,4)--(4,4);
\draw (0,2)--(4,2);
\draw (0,0)--(4,0);
\draw (0,12)--(0,0);
\draw (2,12)--(2,0);
\draw (4,12)--(4,0);
\draw (6,12)--(6,8);
\draw (8,12)--(8,8);
\draw (10,12)--(10,8);
\node at (1,11) {1};
\node at (1,9) {6};
\node at (1,7) {11};
\node at (1,5) {13};
\node at (1,3) {15};
\node at (1,1) {17};
\node at (3,11) {2};
\node at (3,9) {7};
\node at (3,7) {12};
\node at (3,5) {14};
\node at (3,3) {16};
\node at (3,1) {18};
\node at (5,11) {3};
\node at (5,9) {8};
\node at (7,11) {4};
\node at (7,9) {9};
\node at (9,11) {5};
\node at (9,9) {10};
\node at (-2,6){$t=$};

\end{tikzpicture}
\end{center}
\caption{}
\label{fig:a}
\end{figure}
\vspace{.2cm}

\noindent then $R(t)=S_{\{1,2,3,4,5\}}\times S_{\{6,7,8,9,10\}}\times S_{\{11,12\}}\times S_{\{13,14\}}\times S_{\{15,16\}}\times S_{\{17,18\}}$ and $H(t)$ is the subgroup generated by the permutations $$(3,4,5)(8,9,10) , (3,4)(8,9)\ \text{and}\ (1,2)(6,7)(11,12)(13,14)(15,16)(17,18).$$
Considered as an abstract group we have that $H(t)\cong S_3\times S_2$.

\section{A New Lower Bound}
The main goal of this section is to prove Theorem \ref{T1}.
 In order to do this, we need now to introduce some new notation and definitions. 
Let $\lambda$ be a partition of $n$ of the form $$\lambda=(\lambda_1^{l_1},\ldots, \lambda_s^{l_s},\mu_1^{m_1},\ldots,\mu_k^{m_k})$$
with $\lambda_i>\lambda_{i+1}$ for all $i\in\{1,\ldots,s\}$, $\mu_j>\mu_{j+1}$ for all $j\in\{1,\ldots,k-1\}$ and with $\lambda_s>\mu_1$. We will call $$\mu=(\mu_1^{m_1},\ldots,\mu_k^{m_k})$$ a \textit{splitting partition} of $\lambda$. Let $\eta$ be the partition defined by $\eta=((\lambda_1-\mu_1)^{l_1},\ldots,(\lambda_s-\mu_1)^{l_s}).$ 

Given any $\lambda$-tableau $t$, we introduce below some definitions that will be very important in order to prove Theorem \ref{T1}.

\begin{itemize}
\item Denote by $z_{\mu}(t)$ the $\mu$-tableau formed by the last $m_1+m_2+\cdots +m_k$ rows of $t$. Clearly $\mathrm{supp}(z_{\mu}(t))\subseteq\{1,2,\ldots,n\}$. 

\item Let $Z_{\mu}(t)$ be the subgroup of $C(t)$ that fixes all the numbers outside the support of $z_{\mu}(t)$ and permutes the rows of equal length of $z_{\mu}(t)$ as blocks for its action. 
Notice that $Z_{\mu}(t)$ is isomorphic as an abstract group to $S_{m_1}\times S_{m_2}\times\cdots\times S_{m_k}$. 

\item Denote by $u_{\mu}(t)$ the $\eta$-tableau formed by the last $\lambda_1-\mu_1$ columns of $t$.

\item Let $U_{\mu}(t)$ be the subgroup of $R(t)$ fixing all the numbers outside the support of $u_{\mu}(t)$ and permuting the columns of $u_{\mu}(t)$ as blocks for its action. Notice that $U_{\mu}(t)$ is a subgroup of $H(t)$. 

\item Finally, denote by $L_{\mu}(t)$ the group $U_{\mu}(t)\times Z_{\mu}(t)$. 

\end{itemize}
\begin{esem}
Let $\lambda=(5,5,2,2,2,2)$ and let $t$ be the $\lambda$-tableau shown in Figure \ref{fig:a} then choosing the splitting partition $\mu$ to be $\mu=(2,2,2,2)$, we have that $U_\mu(t)$ and $Z_\mu(t)$ are the subgroups of $S_{18}$ defined by 
$$U_\mu(t)=\left\langle(3,4,5)(8,9,10) , (3,4)(8,9)\right\rangle$$ and 
$$Z_\mu(t)=\left\langle(11,13,15,17)(12,14,16,18) , (11,13)(12,14)\right\rangle.$$
In particular we have that, as an abstract group, $L_\mu(t)\cong S_3\times S_4$.
\end{esem}

We observe that, by construction, $L_\mu(t)$ permutes both rows and columns of $t$ as blocks for its action. Notice also that $L_\mu(t)$ depends on the initial choice of $\mu$. In general $L_\mu(t)$ and $L_\nu(t)$ are non-isomorphic if $\mu$ and $\nu$ are different splitting partitions of $\lambda$ (see for instance Example \ref{ZZZ}).

 On the other hand, given $\mu$ a splitting partition of $\lambda$ and $u$, $t$ two $\lambda$-tableaux, we have that $L_\mu(t)$ is a conjugate of $L_\mu(u)$ in $S_n$. 

The following lemma is a fundamental step towards the proof of Theorem \ref{T1}. 

\begin{lemma}\label{L1}
Let $\lambda$ be a partition of $n$, $\mu$ a splitting partition of $\lambda$ and $t$ a $\lambda$-tableau. Let $P$ be a Sylow $p$-subgroup of $L_{\mu}(t)$. Then $$e_ty=e_t$$ for all $y\in P$.
\end{lemma}

\begin{proof}

Since $P$ is a Sylow $p$-subgroup of $L_{\mu}(t)=U_{\mu}(t)\times Z_{\mu}(t)$, there exist $P_U$ and $P_Z$ Sylow $p$-subgroups of $U_{\mu}(t)$ and $Z_{\mu}(t)$ respectively such that $P=P_U\times P_Z$. Therefore for any element $y\in P$ there exist unique $h\in P_U$ and $k\in P_Z$ such that $y=hk$. 
Hence it suffices to prove that $e_th=e_t$ for all $h\in P_U$ and $e_tk=e_t$ for all $k\in P_Z$. 

Let $h\in P_U\leq U_{\mu}(t)\leq R(t)$. By definition $h$ permutes the columns of $t$ as blocks for its action, therefore $C(t)^h=C(t)$ and of course 
$\{t\}h=\{t\}$. Hence $$e_th=\sum_{g\in C(t)}\mathrm{sgn}(g)\{t\}gh =\sum_{x\in C(t)^h}\mathrm{sgn}(x)\{t\}hx =\sum_{x\in C(t)}\mathrm{sgn}(x)\{t\}x =e_t.$$
Let $k\in P_Z$. If the characteristic $p$ of the underlying field $\mathbb{F}$ is $2$ then clearly $e_tk=e_t$ since $k\in C(t)$. On the other hand, if the prime characteristic $p>2$ then by definition $P_Z\leq C(t)\cap A_n$, hence we have again that $e_tk=e_t$, as required.
\end{proof}

To proceed with the proof of Theorem \ref{T1} we will denote by $t^\star$ the greatest standard $\lambda$-tableau in the dominance order. In particular
if $\lambda=(\rho_1,\ldots,\rho_k)$, we have that the entries in the $i^{th}$ row of $t^\star$ are 
$$R_i=\{\rho_1+\cdots +\rho_{i-1}+1,\ldots,\rho_1+\cdots +\rho_{i}\},$$
for all $i\in\{1,2,\ldots,k\}$. For example, the standard tableau shown in Figure \ref{fig:a} is the most dominant $(5^2,2^4)$-tableau. 

\begin{proof}[Proof of Theorem~\ref{T1}]

Since the subgroups $L_\mu(t)$ for different tableau $t$ are all conjugate in $S_n$, without loss of generality, it suffices to prove that a Sylow $p$-subgroup of $L_\mu(t^\star)$ is contained in a vertex of $S^\lambda$. Let $P$ be a Sylow $p$-subgroup of $L_\mu(t^\star)$. Lemma \ref{L1} implies that $e_{t^\star}\in (S^\lambda)^P$.
In order to complete the proof we just need to show that $e_{t^\star}\notin \mathrm{Tr}^P(S^\lambda)$.
 Consider $V$ to be the subspace of $S^\lambda$ generated by all the elements of the form $$e_s+e_sg+\cdots +e_sg^{p-1},$$
where $s$ is any standard $\lambda$-tableau and $g$ is any element of $P$.
Since any maximal subgroup of $P$ has index $p$ in $P$, we have that
$\mathrm{Tr}^P(S^\lambda)\leq V$, therefore it will suffice to show that $e_{t^\star}\notin V$.
Suppose by contradiction that $$e_{t^\star}=\sum_{s, g}a_{s,g}(e_s+\cdots +e_sg^{p-1}),\ \ \text{for some $a_{s,g}\in\mathbb{F}$}.$$
Then there exist a standard tableau $s$ such that, when $e_s+\cdots +e_sg^{p-1}$ is expressed as a linear combination of standard polytabloids, $e_{t^\star}$ appears with non zero coefficient. This implies that exists $i\in \{0,1,\ldots,p-1\}$ such that $e_{t^\star}$ appears in the expression of $e_sg^i$. Let $u$ be the column-standard tableau whose columns are setwise equal to the columns of $sg^i$. 
Clearly $e_{sg^i}=\pm e_u$ and
by Theorem \ref{Standard} we have that $e_u=e_{\overline{u}}+x$, where $\overline{u}$ is the row-straightening of $u$ and $x$ is a linear combination of standard polytabloids $e_v$ with $v\triangleleft \overline{u}\trianglelefteq t^\star$. We deduce that $t^\star=\overline{u}$ because 
$t^\star$ is the greatest standard tableau in the dominance order and the standard polytabloids are linearly independent by Theorem \ref{Standard}.

 Observe that if $a,b\in\{1,2,\ldots,n\}$ are in the same row of $t^\star$ then they lie in the same row of $u$ and since the columns of $u$ agree setwise with the columns of $sg^i$ we obtain that necessarily $a$ and $b$ lie in different columns of $sg^i$.
Suppose now by contradiction that there exist $a, b\in\{1,2,\ldots,n\}$ lying in the same row $R_j$ of $t^\star$ and lying also in the same column of $s$. 
Since $g^i\in P\leq L_{\mu}(t^\star)$ permutes the rows of $t^\star$ as blocks for its action, we have that $ag^i, bg^i$ belong to the same row $R_jg^i$ of $t^\star$, and clearly we obtain a contradiction because $ag^i, bg^i$ would then lie in the same column of $sg^i$. 
We just proved that no two numbers in the same row of $t^\star$ can possibly lie in the same column of $s$. Since $s$ is standard, we obtain that $s=t^\star$. Therefore $e_s=e_{t^\star}$ and by Lemma \ref{L1} $e_{t^\star}g=e_{t^\star}$, and so $$e_{t^\star}+e_{t^\star}g+\cdots +e_{t^\star}g^{p-1}=pe_{t^\star}=0.$$
This contradicts our initial assumption. Therefore $e_{t^\star}\notin V$, as required. 
We just proved that $S^\lambda(P)\neq 0$ and therefore we have that $P$ is contained in a vertex of $S^\lambda$. 
\end{proof}

Theorem \ref{T1} generalizes Theorem \ref{mw}. In particular we observe that given $\lambda$ a partition of $n$ such that $S^\lambda$ is an indecomposable Specht module, then choosing the splitting partition $\mu$ to be the empty partition, we have that
$$L_\emptyset(t)= H(t),$$
therefore we obtain Theorem \ref{mw} as a corollary of Theorem \ref{T1}.
Moreover, it is easy to realize that for a large number of partitions $\lambda$ of $n$ there exists a splitting partition $\mu\neq\emptyset$ such that 
$H(t)$ is isomorphic to a proper subgroup of $L_\mu(t)$.
In all these cases our Theorem \ref{T1} strictly improves the lower bound on a vertex of $S^\lambda$ given by 
Wildon. 
One explicit example is $\lambda=(5,5,2,2,2,2)$, $\mu=(2,2,2,2)$ and $t$ the $\lambda$-tableau shown in Figure \ref{fig:a}.

%
%
%

In the following remark we show that in the case of hook-partitions, our Theorem gives a complete description of the vertices of the correspondent Specht modules.

\begin{osserv}
Let $\lambda=(n-k,1^k)$ be an hook partition of a natural number $n$ such that $p$ does not divide $n$. If we choose the splitting partition $\mu$ to be equal to $(1^k)$ then the lower bound on the vertex of $S^\lambda$ obtained from Theorem \ref{T1} is attained. In fact in such case we have 
$$L_\mu(t)\cong S_k\times S_{n-k-1}.$$
By \cite{MW1} we have that the vertex of $S^\lambda$ is isomorphic to a Sylow $p$-subgroup of $S_k\times S_{n-k-1}$.
\end{osserv}

In general it will be very important to analyse the properties of the whole set of subgroups $L_\mu(t)$ for any $\mu$ splitting partition of $\lambda$. 
As shown in the example below, this will allow us to deduce more precise information about the structure of the vertices of $S^\lambda$.

\begin{esem}\label{ZZZ}
Let $p=3$, $n=36$ and $\lambda=(6^3,3^6)$. Denote by $Q$ a vertex of $S^\lambda$ and let $t$ be the $\lambda$-tableau shown in Figure 2 below.
Applying Theorem \ref{T1} with respect to the splitting partition $\mu=(3^6)$ of $\lambda$, we obtain that $Q$ contains a subgroup isomorphic to the elementary abelian $3$-group of order $27$. In particular, a Sylow $3$-subgroup of $L_\mu(t)$ is generated by the permutations $$\sigma:=(4,5,6)(10,11,12)(16,17,18) ,\ \rho:=(19,22,25)(20,23,26)(21,24,27)$$ $$\text{and}\ \tau:=(28,31,34)(29,32,35)(30,33,36).$$
Without loss of generality we can suppose that the group $R:=\left\langle\sigma,\rho,\tau\right\rangle$ is contained in $Q$. 

In this case we can say something even more precise. 
In fact, we observe that a Sylow $3$-subgroup of $L_\emptyset (t)\cong S_3\times S_3$ has support of full size equal to $36$. 
In particular it contains the element $\gamma$ defined by $$\gamma:=(1,2,3)(7,8,9)(13,14,15)(19,20,21)\cdots (34,35,36)(4,5,6)(10,11,12)(16,17,18).$$
Hence a conjugate $\tilde{\gamma}$ of $\gamma$ is contained in $Q$.
Since $\tilde\gamma$ has support of size $36$ it follows that $\left\langle R,\tilde\gamma\right\rangle$ is a subgroup of $Q$, strictly containing $R$, of order $81$.

\end{esem}

\begin{figure}[!htp]
\begin{center}

\begin{tikzpicture}[scale=0.3, every node/.style={transform shape}]
\tikzstyle{every node}=[font=\large]

\draw (0,14)--(12,14);
\draw (0,12)--(12,12);
\draw (0,10)--(12,10);
\draw (0,8)--(12,8);
\draw (0,6)--(6,6);
\draw (0,4)--(6,4);
\draw (0,2)--(6,2);
\draw (0,0)--(6,0);
\draw (0,-2)--(6,-2);
\draw (0,-4)--(6,-4);
\draw (0,14)--(0,-4);
\draw (2,14)--(2,-4);
\draw (4,14)--(4,-4);
\draw (12,14)--(12,8);
\draw (6,14)--(6,-4);
\draw (6,14)--(6,8);
\draw (8,14)--(8,8);
\draw (10,14)--(10,8);
\node at (1,13) {1};
\node at (1,11) {7};
\node at (1,9) {13};
\node at (1,7) {19};
\node at (1,5) {22};
\node at (1,3) {25};
\node at (1,1) {28};
\node at (1,-1) {31};
\node at (1,-3) {34};
\node at (3,13) {2};
\node at (3,11) {8};
\node at (3,9) {14};
\node at (3,7) {20};
\node at (3,5) {23};
\node at (3,3) {26};
\node at (3,1) {29};
\node at (3,-1) {32};
\node at (3,-3) {35};
\node at (5,13) {3};
\node at (5,11) {9};
\node at (5,9) {15};
\node at (5,7) {21};
\node at (5,5) {24};
\node at (5,3) {27};
\node at (5,1) {30};
\node at (5,-1) {33};
\node at (5,-3) {36};
\node at (7,13) {4};
\node at (7,11) {10};
\node at (7,9) {16};
\node at (9,13) {5};
\node at (9,11) {11};
\node at (9,9) {17};
\node at (11,13) {6};
\node at (11,11) {12};
\node at (11,9) {18};

\node at (-2,6){$t=$};

\end{tikzpicture}
\end{center}
\caption{}
\label{fig:b}
\end{figure}
\vspace{.2cm}

\medskip

\section*{Acknowledgements}
I would like to thank my PhD supervisor Dr.~ Mark Wildon for his constant support and his helpful comments.

\def\cprime{$'$} \def\Dbar{\leavevmode\lower.6ex\hbox to 0pt{\hskip-.23ex
  \accent"16\hss}D}
\providecommand{\bysame}{\leavevmode\hbox to3em{\hrulefill}\thinspace}
\providecommand{\MR}{\relax\ifhmode\unskip\space\fi MR }
\providecommand{\MRhref}[2]{%
  \href{http://www.ams.org/mathscinet-getitem?mr=#1}{#2}
}
\providecommand{\href}[2]{#2}
\renewcommand{\MR}[1]{\relax}

\end{document}